\documentclass[11pt, a4paper]{amsart}%{article}
\usepackage{amsmath,amssymb}
\usepackage{amsthm}
\usepackage[utf8]{inputenc}
\usepackage[pdftex]{graphicx}%[dvips]{graphicx}%
\DeclareGraphicsExtensions{.png,.pdf,.jpg}
\usepackage{color}
\let\abs=\envert

\let\norm=\enVert

\newtheorem{teor}{Theorem}[section]
\newtheorem{corol}[teor]{Corollary}
\theoremstyle{definition}
\newtheorem{defi}[teor]{Definition}
\newtheorem{lema}[teor]{Lemma}

\numberwithin{equation}{section}

\begin{document}
\title[Absolute divergence in the infinite torus]{On the absolute divergence of Fourier series in the infinite dimensional torus}
%%%%%%%%%%%%%%%%%%
\author[E. Fern\'andez]{Emilio Fern\'{a}ndez} 
\address{Departamento de Matem\'{a}ticas y Computaci\'{o}n, 
Universidad de La Rioja, c/ Madre de Dios, 53, 26006 Logro\~no, Spain.}
\email{emfernan@unirioja.es}

\author[L. Roncal]{Luz Roncal}
\address{BCAM -- Basque Center for Applied Mathematics,
Alameda de Mazarredo, 14
E-48009 Bilbao, Basque Country, Spain.}
\email{lroncal@bcamath.org}

%\thanks{}



\begin{abstract}
We present some simple counterexamples, based on quadratic forms in infinitely many variables, showing that the implication
$f\in C^{(\infty}(\mathbb{T}^\omega)\Longrightarrow\sum_{\bar{p}\in\mathbb{Z}^\infty}|\widehat{f}(\bar{p})|<\infty$ is false. There are functions of the class $C^{(\infty}(\mathbb{T}^\omega)$ (depending on an infinite number of variables) whose Fourier series diverges absolutely. This fact establishes a significant difference from the finite dimensional case.
\end{abstract}

%%%%%%%%%%%%%%%%%%%%%%%%%%%%%%%%%%%%%%%%%%%%%%%%%%%%%%%%%%%%%%%%%%%%%%%
\maketitle
%%%%%%%%%%%%%%%%%%%%%%%%%%%%%%%%%%%%%%%%%%%%%%%%%%%%%%%%%%%%%%%%%%%%%%%
\thispagestyle{empty}

\section{Introduction}
%%%%%%%%%%%%%%%%%%%%%%%%%%%%%%%%%%%%%%%%%%%%%%%%%%%%%%%%%%%%%%%%%%%%%%%%%%%%
The following result establishes a sufficient condition of smoothness on a function defined on the $n$-dimensional torus $\mathbb{T}^n$ ($n\ge 1$) for the absolute convergence of its Fourier series:
\begin{teor} \emph{(\cite[p. 249]{steinweiss}.)}
If $f\in C^{(k}(\mathbb{T}^n)$, $k>n/2$, then 
$$
\sum_{m\in\mathbb{Z}^n}\big|\widehat{f}(m)\big|<\infty.
$$
\end{teor}

When $f\in C^{(\infty}(\mathbb{T}^n)$, more conclusive results hold, for example (see \cite[Th. 7.25, p. 202]{rudinfa}):
\begin{teor}
\label{teor:rudinfa}
If $f\in C^{(\infty}(\mathbb{T}^n)$, then
$$\sum_{m\in\mathbb{Z}^n}(1+|m|)^N\big|\widehat{f}(m)\big|<\infty\quad\forall N=0,1,\ldots,\quad |m|=\Bigl(\sum_{i=1}^n m_i^2\Bigr)^{1/2}. $$ 
\end{teor}

This same result holds for \emph{cylindrical} infinitely smooth functions defined on the infinite dimensional torus $\mathbb{T}^\omega$, the compact abelian group which is the complete direct sum of countably many copies of $\mathbb{T}\simeq\mathbb{R}/\mathbb{Z}$. Recall that
$f$ is a \textit{cylindrical function} on $\mathbb{T}^\omega$ if $f$ depends only on a finite number of variables, i.e., there exists $n\ge 1$ and  $g_n\colon \Omega_n\to\mathbb{C}$, with $\Omega_n\subseteq\mathbb{T}^n$, such that $f=g_n\circ\pi_n$, with $\pi_n\colon\mathbb{T}^\omega\to\mathbb{T}^n$ being the canonical projection.  The space of cylindrical functions of class $C^{(\infty}$ on $\mathbb{T}^\omega$ (see Definition \ref{def:dif}) is defined (\cite[p. 73--75]{bendikov}) by
$$
\mathcal{D}(\mathbb{T}^\omega)=\bigcup_{n=1}^\infty\bigl\{ g_n\circ\pi_n \vert g_n\in C^{(\infty}(\mathbb{T}^n) \bigr\}
$$
so that if  $f\in\mathcal{D}(\mathbb{T}^\omega)$, then there exist  $p\in\mathbb{N}$ and $g_p\in C^{(\infty}(\mathbb{T}^p)$ such that $f=g_p\circ\pi_p$. 

The dual group of $\mathbb{T}^\omega$, denoted by $\mathbb{Z}^\infty$, is the direct sum of countably many copies of $\mathbb{Z}$, formed by the finitely nonzero sequences of integer numbers. Denote by $dx$ the normalized Haar measure on $\mathbb{T}^\omega$. If $f\in L^1(\mathbb{T}^\omega)$, then the function $\widehat{f}$  defined on $\mathbb{Z}^\infty$ by
$$
\widehat{f}(\bar{n})=\int_{\mathbb{T}^\omega}f(x)e^{-2\pi i \bar{n}\cdot x }\,dx\qquad( \bar{n}\in\mathbb{Z}^\infty)
$$ 
is the \textit{Fourier transform} of $f$, the \textit{Fourier series} of $f$ being the formal series (observe that $\mathbb{Z}^\infty$ is a countable set)

$$
\sum_{\bar{n}\in\mathbb{Z}^\infty} \widehat{f}(\bar{n})e^{2\pi i \bar{n}\cdot x }.
$$

By using the ideas in the proof of Theorem \ref{teor:rudinfa}, the following result can be proved (see also \cite[Proposition 1]{bendikov86}):
\begin{teor} 
If $\phi\in\mathcal{D}(\mathbb{T}^\omega)$, then 
$$
\sum_{\bar{p}\in\mathbb{Z}^\infty}(1+|\bar{p}|)^N\big|\widehat{\phi}(\bar{p})\big|<\infty\quad\forall N=0,1,\ldots,
\quad |\bar{p}|=\Bigl(\sum_{i=1}^\infty p_i^2\Bigr)^{1/2}. 
$$ 
\end{teor}

In May 2016, in a private communication to the second author, Professor  A. D. Bendikov conjectured that the implication $f\in C^{(\infty}(\mathbb{T}^\omega) \Rightarrow \sum_{\bar{p}\in\mathbb{Z}^\infty}|\widehat{f}(\bar{p})|<\infty$, which holds, as already mentioned, for functions depending only on a finite number of variables, is in general false. This would mean a significant difference to what happens in the finite-dimensional case. 

In this note we confirm Bendikov's conjecture by producing some counterexamples via quadratic forms depending on an infinite number of variables. The construction of such counterexamples is based on classical results of Toeplitz \cite{toeplitz1913}, Littlewood \cite{little1930} and Bohnenblust and Hille \cite{bohnen1931}\footnotemark.

\footnotetext{Prof. Bendikov suggested that a counterexample could be constructed through an appropriate Jacobi theta function in an infinite number of variables. Our construction is different.}

 The main result in this note is the following.
\begin{teor} 
\label{teor:main}
There exist functions of the class $C^{(\infty}(\mathbb{T}^\omega)$ (depending on an infinite number of variables) whose Fourier series diverges absolutely. 
\end{teor}

Although we restrict ourselves to the case of the infinite dimensional torus, we point out that Bendikov and L. Saloff-Coste \cite{bendikov2005} have studied several scales of smooth functions in the more general setting of connected infinite-dimensional compact groups.

In Section \ref{sec:prelim} we introduce some definitions and give several basic results. In Section~\ref{sec:bilinear} we present a detailed account of bilinear and quadratic forms in an infinite number of variables used to construct our counterexamples. The proof of Theorem~\ref{teor:main} and the counterexamples are given in Section \ref{sec:diverge}.

\section{Premilinary definitions and results}
\label{sec:prelim}

We begin by providing some basic principles.

\begin{defi} (\cite[p. 130]{bohr1925}.) 
The function $f\colon\mathbb{T}^\omega\to\mathbb{C}$ is \emph{continuous at the point} $x^{(0)}=(x_1^{0},x_2^{0},\ldots)$ if for every $\varepsilon>0$ there is a positive integer $m$ and a number $\delta>0$ such that for each 
$(x_1,x_2,\ldots)\in\mathbb{T}^\omega$ satisfying 
$$
\abs{x_{j}-x_{j}^{0}}<\delta \quad (j=1,2,\ldots,m),
$$
we have
$$
\abs{f(x_1,x_2,\ldots)-f(x_1^{0},x_2^{0},\ldots)}<\varepsilon.
$$
Since $\mathbb{T}^\omega$ is compact, the vector space 
$$C^{(0}(\mathbb{T}^\omega)=\left\{f\colon\mathbb{T}^\omega\to\mathbb{C}\:\vert\: f \text{ is continuous at all } x\in\mathbb{T}^\omega\right\}
$$ 
is a Banach space with the norm $\norm{f}_\infty=\max_{x\in\mathbb{T}^\omega}\abs{f(x)}$. 
\end{defi}

\begin{lema} \label{lema14}
Let $\varphi(t)\in C^{(0}(\mathbb{T})$  and $\sum_{j=1}^\infty a_j$ be an absolutely convergent series of complex numbers. Then the function 
$$
\Psi(x)=\sum_{j=1}^\infty a_j\varphi(x_j)
$$ 
is continuous on $\mathbb{T}^\omega$.
\end{lema}
\begin{proof}
We can clearly suppose that $\varphi$ is not the zero function. Fix $x^{(0)}\in \mathbb{T}^\omega$. Given $\varepsilon>0$, since $\sum_{j=1}^\infty \abs{a_j}$ converges, there exists $m_1\in\mathbb{N}$ such that
$$
\sum_{j=m_1+1}^N\abs{a_j}<\frac\varepsilon{4\norm{\varphi}_\infty}\quad \text{for all } N>m_1.
$$ 
On the other hand, for each $j=1,\ldots,m_1$, the continuity of $\varphi$ at $x_j^{(0)}$ ensures the existence of $\delta_j>0$ such that if $\bigl|x_j-x_j^{(0)}\bigr|<\delta_j$, then
$$
\bigl|\varphi(x_j)-\varphi(x_j^{(0)})\bigr|<\frac\varepsilon{2m_1\abs{a_j}}.
$$
Let $\delta=\min_{1\le j\le m_1} \delta_j$. If $x\in\mathbb{T}^\omega$ satisfies $\bigl|x_j-x_j^{(0)}\bigr|<\delta$ for $j=1,\ldots,m_1$, then for all $N>m_1$,
\begin{align*}
\biggl|\sum_{j=1}^N a_j\varphi(x_j)&-\sum_{j=1}^N a_j\varphi(x_j^{(0)})\biggr|
=\biggl|\sum_{j=1}^N a_j\bigl(\varphi(x_j)-\varphi(x_j^{(0)})\bigr)\biggr|\\
&\le\sum_{j=1}^{m_1}  \abs{a_j}\bigl|\varphi(x_j)-\varphi(x_j^{(0)})\bigr|+ \sum_{j=m_1+1}^N  \abs{a_j}\bigl|\varphi(x_j)-\varphi(x_j^{(0)})\bigr|\\
&\le\sum_{j=1}^{m_1}  \abs{a_j}\cdot\frac\varepsilon{2m_1\abs{a_j}}+2\norm{\varphi}_\infty\sum_{j=m_1+1}^N  \abs{a_j}<\frac\varepsilon2+\frac\varepsilon2=\varepsilon.
\end{align*}
Moreover, $\sum_{j=1}^\infty \abs{a_j\varphi(x_j)}\le \norm{\varphi}_\infty\sum_{j=1}^\infty \abs{a_j}$ for all $x\in \mathbb{T}^\omega$. Therefore, the series defining $\Psi(x)$ is absolutely convergent, thus $\Psi(x)$ is defined for all $x\in \mathbb{T}^\omega$ and there exists $m_2=m_2(\varepsilon)$ such that, if $N>m_2$, then
$$
\bigl|\Psi(x)-\sum_{j=1}^N a_j\varphi(x_j)\bigr|<\varepsilon\quad \forall x\in \mathbb{T}^\omega.
$$
Consequently, taking $M=\max\{m_1,m_2\}$, we have 
\begin{align*}
\bigl|\Psi(x)-\Psi(x^{(0)})\bigr|&\le\bigl|\Psi(x)-\sum_{j=1}^M a_j\varphi(x_j)\bigr|+\biggl|\sum_{j=1}^M a_j\varphi(x_j)-\sum_{j=1}^M a_j\varphi(x_j^{(0)})\biggr|\\
&\qquad+\bigl|\Psi(x^{(0)})-\sum_{j=1}^M a_j\varphi(x_j^{(0)})\bigr|
\\&<3\varepsilon
\end{align*} 
if $x\in \mathbb{T}^\omega$ satisfies $\bigl|x_j-x_j^{(0)}\bigr|<\delta$ for $j=1,\ldots,M$, and therefore $\Psi(x)$ is continuous at $x^{(0)}$.
\end{proof}

\begin{defi} 
\label{def:dif}
Let $f$ be a function defined on $\mathbb{T}^\omega$.
For each multiindex $\alpha=(\alpha_1,\alpha_2,\ldots)$ which is finitely nonzero, that is, $\alpha_j\ne 0$ for only finitely many $j$, the \emph{partial differentiation operator} $D^{\alpha}$ is defined by
$$
D^\alpha f= D_{j_1}^{\alpha_{j_1}}\cdots D_{j_m}^{\alpha_{j_m}}f=\frac{\partial^{\alpha_{j_1}}}{\partial x_{j_1}^{\alpha_{j_1}}} \cdots \frac{\partial^{\alpha_{j_m}}}{\partial x_{j_m}^{\alpha_{j_m}}}f
\quad\text{if $\alpha_j=0$ \text{ for } $ j\notin\{j_1,\ldots,j_m\}$.}
$$
The \emph{total order} of $\alpha$ is $\abs{\alpha}=\alpha_{j_1}+\ldots+\alpha_{j_m}$.  When $\abs{\alpha}=0$, $D^\alpha f=f$.

For each $k$, $C^{(k}(\mathbb{T}^\omega)$ is defined as the class of all functions
$f$  with continuous \emph{partial derivatives} up to the $k$-th order, i.e., $D^\alpha f\in C^{(0}(\mathbb{T}^\omega)$ for all finitely nonzero multiindices $\alpha$ such that $\abs{\alpha}\le k$.  With the norm
$$
\norm{f}_{(k)}=\sup_{0\le\abs{\alpha}\le k}\norm{D^\alpha f}_\infty
$$
where $\norm{D^\alpha f}_\infty=\max_{x\in\mathbb{T}^\omega}\abs{(D^\alpha f)(x)}$ for each fixed $\alpha$,
$C^{(k}(\mathbb{T}^\omega)$ is a  Banach space \cite[2.2.4]{edwards}. 
The space of infinitely differentiable functions  is the intersection  $C^{(\infty}(\mathbb{T}^\omega)=\bigcap_{k=0}^\infty C^{(k}(\mathbb{T}^\omega)$ and it is a  Fr\'echet space \cite[12.1]{edwards}.
\end{defi}

\smallskip
\noindent\textbf{Double series.} (See \cite[pp. 72-76]{bromwich}, and also \cite{moricz2011}.) Consider a double series of complex numbers, 
\begin{equation}\label{seriedoble}
\sum_{m,n=1}^\infty a_{mn}.
\end{equation} 
\emph{Rectangular} partial (finite) sums of \eqref{seriedoble} are 
$$s_{MN}:=\sum_{m=1}^M\sum_{n=1}^N a_{mn}, \quad(M,N)\in\mathbb{N}^2.$$ 
The series \eqref{seriedoble} is said to \emph{converge to $s\in\mathbb{C}$ in Pringsheim's sense} when  for every $\varepsilon>0$ there exists $\mu$ such that
$$
|s_{MN}-s|<\varepsilon \quad\text{if $M,N\ge\mu$}.
$$ 

A necessary and sufficient condition for the convergence of \eqref{seriedoble} in Pringsheim's sense is:
\begin{equation}\label{cnysprings}
\forall\varepsilon>0 \ \exists\mu: |s_{PQ}-s_{MN}|<\varepsilon \text{ if } P>M\ge\mu \text{ and } Q>N\ge\mu.
\end{equation}

When the series $\sum_{m,n}a_{mn}$ and $\sum_{m,n}b_{mn}$ converge in Pringsheim's sense, so does $\sum_{m,n}(a_{mn}+b_{mn})$, and
\begin{equation}\label{doblesuma}
\sum_{m,n}(a_{mn}+b_{mn})=\sum_{m,n}a_{mn}+\sum_{m,n}b_{mn}.
\end{equation}

Hardy \cite[p. 88]{hardy17} introduced the notion of regular convergence of a double series as follows: the series \eqref{seriedoble} is said \emph{to converge regularly
to $s\in\mathbb{C}$} if it converges to $s$ in Pringsheim’s sense and, in addition, each of its
row and column series,  $\sum_{n=1}^\infty a_{mn}$ for each $m=1,2,\ldots$, and  $\sum_{m=1}^\infty a_{mn}$ for each $n=1,2,\ldots$,  also converges as a single series.

An absolutely convergent double series is also regularly convergent, and regular convergence is sufficient for
$$
\sum_{m,n=1}^\infty a_{mn}=\sum_{m=1}^\infty \sum_{n=1}^\infty a_{mn}=\sum_{n=1}^\infty \sum_{m=1}^\infty a_{mn}
$$
hold \cite[Th. 1]{moricz2011}.

\section{Bilinear and quadratic forms in an infinite number of variables}
\label{sec:bilinear}

Let us denote by $\mathcal{S}:=\bigl\{ (z_n)_{n=1}^\infty \,|\, z_n\in\mathbb{C},\  \abs{z_n}\le 1 \ \forall n\in\mathbb{N}\bigr\}$ the infinite-dimensional polydisc 
(the closed unit ball of $\ell_\infty(\mathbb{N})$). Analogously to $\mathbb{T}^\omega$, we will consider the space  $\mathcal{S}$ with the topology of the cartesian product of infinitely many closed unit circles of the complex plane. In particular, if $x\in\mathbb{T}^\omega$, then
$$
z=e^{2\pi i x}:=\bigl(e^{2\pi i x_1},\ldots,e^{2\pi i x_n},\ldots \bigr)\in\mathcal{S}.
$$

We define a \emph{bilinear form in $\mathcal{S}$} (in principle only formally) by the expression
\begin{equation}\label{qbili}
Q(x,y):=\sum_{m,n=1}^\infty a_{mn}x_my_n \qquad (a_{mn}\in\mathbb{C},\ x,y\in\mathcal{S}).
\end{equation}
The bilinear character and the very existence of the function $Q(x,y)$ depend on the convergence of the double series above.

\begin{defi} The series  \eqref{qbili} is \emph{completely bounded in $\mathcal{S}$} if there is a constant  $H$ such that
\begin{equation}\label{qmnleh}
\biggl|\sum_{m=1}^M \sum_{n=1}^N a_{mn}x_my_n\biggr|\le H \quad \forall x,y\in\mathcal{S}, \ \forall M,N\in\mathbb{N}.
\end{equation}
\end{defi}
The following property is immediate.

\begin{lema}\label{stepI}
Suppose that the series  \eqref{qbili} is completely bounded in $\mathcal{S}$. Then  the series $\sum_{n=1}^\infty \abs{a_{mn}}$ for each $m\in\mathbb{N}$, and  $\sum_{m=1}^\infty \abs{a_{mn}}$ for each $n\in\mathbb{N}$, are convergent.
\end{lema}
\begin{proof}
For $M,N\in\mathbb{N}$, let $Q_{MN}(x,y)$ denote the rectangular partial sum or section
\begin{equation*}\label{teorlit3}
Q_{MN}(x,y):=\sum_{m=1}^M \sum_{n=1}^N a_{mn}x_my_n\quad (x,y\in\mathcal{S}).
\end{equation*}
The section $Q_{MN}(x,y)$ only depends on the first $M$  components of $x$ and on the  first $N$ components of  $y$, 
and thus we can consider it as a bilinear form on $D^M\times D^N$, where $D$ denotes the closed unit disc of the complex plane. Let us write
\begin{equation*}\label{eq114prima}
x^{(M)}:=(x_1,\ldots,x_M),\quad y^{(N)}:=(y_1,\ldots,y_N).
\end{equation*}
Then by hypothesis we have 
$$
\big|Q_{MN}\bigl(x^{(M)},y^{(N)}\bigr)\big|\le H \quad\text{ if \ $\big\|x^{(M)}\big\|_\infty,\,\big\|y^{(N)}\big\|_\infty\le 1$.}
$$

Fix $n_0\in\mathbb{N}$ (when we fix $m_0\in\mathbb{N}$, we proceed in a similar way). Consider the points
$$
\xi_{n_0}:=\Bigl(\frac{\overline{a_{1n_0}}}{\abs{a_{1n_0}}},\ldots, \frac{\overline{a_{mn_0}}}{\abs{a_{mn_0}}},\ldots\Bigr)\quad
\text{ and }\quad 
\eta_{n_0}:= (\delta_{1n_0},\ldots,\delta_{mn_0},\ldots)
$$
($\delta_{ij}$ is the Kronecker's symbol).
Obviously, $\xi_{n_0}$ and $\eta_{n_0}$ belongs to $\mathcal{S}$, and for each $M\in\mathbb{N}$ such that  $M>n_0$, 
$$
\sum_{m=1}^M \abs{a_{mn_0}}=Q_{MM}\bigl(\xi_{n_0}^{(M)},\eta_{n_0}^{(M)}\bigr)=Q_{MM}(\xi_{n_0},\eta_{n_0})=|Q_{MM}(\xi_{n_0},\eta_{n_0})|\le H
$$
with $H$ independent of $M$. Consequently the series  $\sum_{m=1}^\infty \abs{a_{mn_0}}$ is convergent.
\end{proof}

%%%%%%%%%%%%%%%%%%%%%%%%%%%%%%%%%%%%%%%%%
%%%%%%%%%%%%%%%%%%%%%%%%%%%%%%%%%%%%%%%%%

The theorem which follows is due to Littlewood \cite[p. 166-168]{little1930}. 

\begin{teor} \label{teorlittle} If the series \eqref{qbili} is completely bounded in $\mathcal{S}$ by a constant $H$, then it converges in Pringsheim's sense, uniformly in $\mathcal{S}^2$, to a bilinear form $Q(x,y)$ which satisfies  $|Q(x,y)|\le H$ for all $x,y\in\mathcal{S}$ \emph{(we then say that the bilinear form $Q(x,y)$ is \emph{completely bounded on  $\mathcal{S}$})}.
\end{teor} 

Observe that $Q(x,y)$ is a bilinear form if for all  $ x,x',y,y'\in\mathcal{S}$ it is verified
\begin{equation}\label{bilinealidad}
Q(x,y+y')=Q(x,y)+Q(x,y'),\quad Q(x+x',y)=Q(x,y)+Q(x',y). 
\end{equation}

If the bilinear form $Q(x,y)$ is completely bounded on  $\mathcal{S}$, then in particular, given
$\varepsilon>0$, there exists $\nu_1=\nu_1(\varepsilon)$ such that
$\bigl|Q(x,y)-Q_{\nu\nu}(x,y)\bigr|<\varepsilon$ for all $(x,y)\in\mathcal{S}^2$ and $\nu\ge\nu_1$. From this we easily deduce:

\begin{corol}\label{corolittle} A bilinear form $Q(x,y)$ completely bounded on  $\mathcal{S}$ defines a continuous function  on $\mathcal{S}^2$.
\end{corol}
\begin{proof}
Fix $(x_0,y_0)\in\mathcal{S}^2$ and $\varepsilon>0$. 
First, as just said above, there exists $\nu_1(\varepsilon)$ such that
$$
\bigl|Q(x,y)-Q_{\nu\nu}(x,y)\bigr|<\frac\varepsilon3 \quad \forall(x,y)\in\mathcal{S}^2
$$
if $\nu\ge\nu_1$. On the other hand, the bilinear form defined on $D^{\nu_1}\times D^{\nu_1}$ by 
$$
Q_{\nu_1\nu_1}\bigl(x^{(\nu_1)},y^{(\nu_1)}\bigr)=
\sum_{m=1}^{\nu_1}\sum_{n=1}^{\nu_1}a_{mn}x_my_n
$$
is continuous at $\bigl(x_0^{(\nu_1)},y_0^{(\nu_1)}\bigr)$. Thus, 
there exists 
$\delta>0$ (depending on $(x_0,y_0)$ and $\varepsilon$) such that, if 
$\max_{1\le j\le\nu_1}\bigl\{|x_j-x_{0j}|,|y_j-y_{0j}|\bigr\}<\delta$, then
$$
\bigl| Q_{\nu_1\nu_1}\bigl(x^{(\nu_1)},y^{(\nu_1)}\bigr)-Q_{\nu_1\nu_1}\bigl(x_0^{(\nu_1)},y_0^{(\nu_1)}\bigr) \bigr|
<\frac\varepsilon3.
$$

Consequently,  for every $(x,y)\in\mathcal{S}^2$ with $\max\bigl\{|x_j-x_{0j}|,|y_j-y_{0j}|\bigr\}<\delta$ for $j=1,\ldots,\nu_1$, we have
\begin{align*}
\bigl|Q(x,y)-Q(x_0,y_0)\bigr|&\le
\bigl|Q(x,y)-Q_{\nu_1\nu_1}(x,y)\bigr|
+\bigl|Q_{\nu_1\nu_1}(x,y)-Q_{\nu_1\nu_1}(x_0,y_0)\bigr|\\
&\quad +\bigl|Q_{\nu_1\nu_1}(x_0,y_0)-Q(x_0,y_0)\bigr|\\
&<\frac\varepsilon3+\frac\varepsilon3+\frac\varepsilon3=\varepsilon
\end{align*}
(we have used that $Q_{\nu_1\nu_1}(x,y)=Q_{\nu_1\nu_1}\bigl(x^{(\nu_1)},y^{(\nu_1)}\bigr)$ for all $(x,y)\in\mathcal{S}^2$), and the continuity of $Q(x,y)$ at $(x_0,y_0)$ is proved.
\end{proof}

\noindent\textbf{Some more definitions and remarks.}
Let $Q(x,y)=\sum_{m,n=1}^\infty a_{mn}x_my_n$ be a bilinear form completely bounded by a constant $H$  in $\mathcal{S}$, and which defines, according to Corollary~\ref{corolittle}, a continuous function on $\mathcal{S}^2$. Let 
$$
C(x):=Q(x,x)=\sum_{m,n=1}^\infty a_{mn}x_mx_n
$$
for $x\in\mathcal{S}$. The quadratic form $C(x)$ is also called \emph{completely bounded in $\mathcal{S}$}, because
$$
|C_M(x)|=\biggl|\sum_{m=1}^M\sum_{n=1}^M a_{mn}x_mx_n\biggr|\le H \quad\forall x\in\mathcal{S},\ \forall M\in\mathbb{N}.
$$ 

When the bilinear form $Q(x,y)$ is completely bounded in $\mathcal{S}$, its \emph{partial derivatives} are well defined (see \cite[p. 128]{hilbert1912}). Writing, for each $p\in\mathbb{N}$, $e_p=(\delta_{pn})_{n=1}^\infty\in\mathcal{S}$, and applying \eqref{bilinealidad}, we have
\begin{align*}
\frac{\partial Q}{\partial y_p}(x,y)&=\lim_{t\to 0}\frac{Q(x,y+t e_p)-Q(x,y)}t=Q(x,e_p)=\sum_{m=1}^\infty a_{mp}x_m,
\\
\frac{\partial Q}{\partial x_p}(x,y)&=\lim_{t\to 0}\frac{Q(x+t e_p,y)-Q(x,y)}t=Q(e_p,y)=\sum_{n=1}^\infty a_{pn}x_n,
\end{align*}
and thus these partial derivatives are bounded linear forms. 
According to Lemma~\ref{stepI}, the series
$\sum_{n=1}^\infty \abs{a_{pn}}$ and $\sum_{m=1}^\infty \abs{a_{mp}}$ are convergent
for all $p$, and hence one can deduce that $\frac{\partial Q}{\partial x_p}(x,y)$ and $\frac{\partial Q}{\partial y_p}(x,y)$ are continuous on $\mathcal{S}^2$ by applying a result (on $\mathcal{S}^2$) analogous to Lemma  \ref{lema14} (in $\mathbb{T}^\omega$).

\begin{corol}\label{cuadratica}
\emph{(a)} If the bilinear form $Q(x,y)$ is completely bounded on $\mathcal{S}$, then
the quadratic form $C(x)=Q(x,x)$ belongs to the class $C^{(\infty}(\mathcal{S})$. 

\emph{(b)} If the quadratic form $C(x)=Q(x,x)$ is completely bounded in $\mathcal{S}$, then it belongs to the class $C^{(\infty}(\mathcal{S})$. 
\end{corol}

\begin{proof} (a) The quadratic form $C(x)=Q(x,x)$ is continuous in $\mathcal{S}$ according to Corollary \ref{corolittle}. For each $p\in\mathbb{N}$ we have
\begin{align*}
\frac{\partial C}{\partial x_p}(x)&=\frac{\partial Q}{\partial x_p}(x,x)+\frac{\partial Q}{\partial y_p}(x,x)=\sum_{n=1}^\infty a_{pn}x_n+\sum_{m=1}^\infty a_{mp}x_m\\
&=\sum_{j=1}^\infty (a_{pj}+a_{jp})x_j
\end{align*}
due to the absolute convergence of each series. Then, by applying Lemma~\ref{lema14}, the linear form $\frac{\partial C}{\partial x_p}(x)$ is continuous  on $\mathcal{S}$, and its partial derivatives are constant functions.

(b) From the identity
$$
Q(x,y)=Q\bigl(\tfrac12(x+y), \tfrac12(x+y)\bigr)-Q\bigl(\tfrac12(x-y), \tfrac12(x-y)\bigr)
$$
it follows that the bilinear form $Q(x,y)$ is completely bounded on $\mathcal{S}$. Then apply part (a).
\end{proof}

%%%%%%%%%%%%%%%%%%%%%%%%%%%%%%%%%%%%%%%%%
\section{Functions in $C^{(\infty}(\mathbb{T}^\omega)$ whose Fourier series diverges absolutely} 
\label{sec:diverge}

In this section, we prove Theorem \ref{teor:main}. In 1913, Toeplitz \cite[p. 427]{toeplitz1913} introduced a quadratic form  
\begin{equation}
\label{eq:quad}
C(z)=\sum_{m,n=1}^\infty a_{mn}z_mz_n \quad (z\in\mathcal{S})
\end{equation} 
in infinitely many variables, symmetric (i.e., $a_{mn}=a_{nm}$), completely bounded on $\mathcal{S}$ in the above sense, and such that the series $\sum_{m,n=1}^\infty |a_{mn}|$ diverges. This quadratic form will be described below. We will simply replace 
$z=e^{2\pi ix}$ (i.e., $z_j=e^{2\pi ix_j}$ for all $j$) with $x\in\mathbb{T}^\omega$ in Toeplitz's form, and consider the function
\begin{equation}\label{Ftoeplitz}
F(x)=C(e^{2\pi ix})=\sum_{m,n=1}^\infty a_{mn}e^{2\pi i (x_m+x_n)},\quad x=(x_j)_{j=1}^\infty\in\mathbb{T}^\omega.
\end{equation}
From Corollary \ref{cuadratica} (b) it follows that $F\in C^{(\infty}(\mathbb{T}^\omega)$. In particular, $F$ is integrable.

Let us now calculate the Fourier coefficients of $F$. We will use $F(x)=\lim_{M\to\infty}F_M(x)$, where
$$F_M(x)=\sum_{m,n=1,\ldots,M} a_{mn}e^{2\pi i (x_m+x_n)}.
$$
Since the quadratic form \eqref{eq:quad} is completely bounded on $\mathcal{S}$, i.e., $|C(z)|\le H$ for all $z\in \mathcal{S}$, we have $\abs{F_M(x)}\le H$  for all $M\in\mathbb{N}$ and $x\in\mathbb{T}^\omega$. This allows to apply Vitali's convergence theorem to write, for any $\bar{p}\in\mathbb{Z}^\infty$ fixed,

\begin{align*}
\widehat{F}(\bar{p})&=\int_{\mathbb{T}^\omega}\Bigl(\sum_{m,n=1}^\infty a_{mn}e^{2\pi i(x_m+x_n)}  \Bigr)e^{-2\pi i \bar{p}\cdot x}\,dx\\
&=\int_{\mathbb{T}^\omega}\Bigl(\lim_{M\to\infty}\sum_{m,n=1,\ldots,M} a_{mn}e^{2\pi i(x_m+x_n)}  \Bigr)e^{-2\pi i \bar{p}\cdot x}\,dx\\
&=\int_{\mathbb{T}^\omega}\lim_{M\to\infty}\Bigl(\sum_{m,n=1,\ldots,M} a_{mn}e^{2\pi i(x_m+x_n)}e^{-2\pi i \bar{p}\cdot x}  \Bigr)\,dx\\
&=\lim_{M\to\infty}\sum_{m,n=1,\ldots,M} a_{mn}\int_{\mathbb{T}^\omega}e^{2\pi i((x_m+x_n)-\bar{p}\cdot x)}\,dx
\end{align*}

\begin{align*}
\phantom{\widehat{F}(\bar{p})}&\overset{\footnotemark}{=}\sum_{m,n=1}^\infty a_{mn }\int_{\mathbb{T}^\omega}e^{2\pi i((x_m+x_n)-\bar{p}\cdot x)}\,dx\\
&=
\begin{cases}
a_{mn}+a_{nm}=2a_{mn}&\text{if $\bar{p}=\bar{e}_m+\bar{e}_n$, $m\ne n$}\\
a_{mm}&\text{if $\bar{p}=2\bar{e}_m$,}\\ 0&\text{otherwise,}  
\end{cases}
\end{align*}
where we denote by $\bar{e}_q$  the element $(\delta_{qj})_{j=1}^\infty$ belonging to $\mathbb{Z}^\infty$.

\footnotetext{In fact, denoting by $M_0$ the \emph{greatest  nonzero index} of $\bar{p}$, i.e., $p_j=0$ for all $j>M_0$, we have 
$$
\sum_{m,n=1}^\infty a_{mn }\int_{\mathbb{T}^\omega}e^{2\pi i((x_m+x_n)-\bar{p}\cdot x)}\,dx
=\sum_{m,n=1,\ldots,M_0} a_{mn}\int_{\mathbb{T}^\omega}e^{2\pi i((x_m+x_n)-\bar{p}\cdot x)}\,dx.
$$}

Thus, the expression \eqref{Ftoeplitz}, which defines $F(x)$, is indeed its Fourier series,
$\sum_{\bar{p}\in\mathbb{Z}^\infty} \widehat{F}(\bar{p})e^{2\pi i\bar{p}\cdot x}$.
%%%%%%%%%%%%%%%%%%%%%%%%%%%%%%%%%%%%%%%%%%%%
%%%%%%%%%%%%%%%%%%%%%%%%%%%%%%%%%%%%%%%%%%%%%
Therefore 
$$
\sum_{\bar{p}\in\mathbb{Z}^\infty}\bigl|\widehat{F}(\bar{p})\bigr|=\sum_{m,n=1}^\infty \abs{a_{mn}},
$$
and since $\sum_{m,n=1}^\infty \abs{a_{mn}}=\infty$, our function $F$ is a counterexample showing that the implication  $f\in C^{(\infty}(\mathbb{T}^\omega)\Longrightarrow  \sum_{\bar{p}\in\mathbb{Z}^\infty}|\widehat{f}(\bar{p})|<\infty$ is false.

Let us proceed to describe the quadratic form $C(z)$. We first show an auxiliary lemma.  Toeplitz \cite[p. 423-426]{toeplitz1913} gave it for real orthogonal matrices. In what follows, $D$ denotes the closed unit disc of the complex plane.

\begin{lema}(Littlewood, \cite[p. 171]{little1930}; see also \cite[p. 609]{bohnen1931}.) \label{formaunit}
Let $A=(a_{mn})_{N\times N}$ be a unitary matrix, i.e., 
$$
\sum_{n=1}^N a_{rn}\overline{a_{sn}}=\delta_{rs} \quad \forall r,s=1,\ldots, N,
$$
 and define $Q_{NN}(x):=N^{-1}\sum_{m,n=1}^N a_{mn}x_mx_n$ for $x\in D^N$. Then,  
$$|Q_{NN}(x)|\le 1\quad \forall x\in D^N.$$
\end{lema}

\noindent\textbf{Toeplitz's quadratic form.}
Toeplitz begins by defining  $C_1(z_1,\ldots,z_4)$ as the quadratic form in $D^4$ whose coefficient matrix is
$$
C_1=\begin{pmatrix}-1&1&1&1\\1&-1&1&1\\1&1&-1&1\\1&1&1&-1 \end{pmatrix}.
$$
The real symmetric matrix $C_1$ verifies $C_1^2=4I$, and so, by Lemma ~\ref{formaunit},
$$\big|C_1(z_1,\ldots,z_4)\big|\le 4^{3/2}=8$$ in $D^4$ (this maximum value is attained for $z_1=\ldots=z_4=1$). 

Next, he defines
$C_2(z_{1},\ldots,z_{4^2})$ as the quadratic form in $D^{4^2}$ whose coefficient matrix is
$$
C_2=\begin{pmatrix}-C_1&C_1&C_1&C_1\\C_1&-C_1&C_1&C_1\\C_1&C_1&-C_1&C_1\\C_1&C_1&C_1&-C_1 \end{pmatrix}.
$$
Lemma \ref{formaunit} yields $$\big|C_2(z_{1},\ldots,z_{4^2})\big|\le (4^2)^{3/2}= 8^2$$ in $D^{4^2}$ (and the maximum modulus is attained for $z_1=\ldots=z_{4^2}=1$). 

Inductively, from the quadratic form in $4^\alpha$  variables ($\alpha\ge 1$) with matrix $C_\alpha$,  one can construct the quadratic form in $4^{\alpha+1}$ variables with matrix
$$
C_{\alpha+1}=\begin{pmatrix}-C_\alpha&C_\alpha&C_\alpha&C_\alpha\\
C_\alpha&-C_\alpha&C_\alpha&C_\alpha\\C_\alpha&C_\alpha&-C_\alpha&C_\alpha\\
C_\alpha&C_\alpha&C_\alpha&-C_\alpha \end{pmatrix}.
$$
According to Lemma \ref{formaunit} we have that, for all $\alpha\in\mathbb{N}$, 
$$\big|C_\alpha(z_{1},\ldots,z_{4^\alpha})\big|\le (4^\alpha)^{3/2}=8^\alpha$$
in $D^{4^\alpha}$.
Finally, for $x\in\mathcal{S}$, Toeplitz defines
\begin{align}\label{formatoe}
C(x)=\frac{\mu_1}8&C_1(x_1,\ldots,x_4)+\frac{\mu_2}{8^2}C_2(x_{4+1},\ldots,x_{4+4^2})
\\&+\frac{\mu_3}{8^3}C_3(x_{4^2+4+1},\ldots,x_{4^2+4+4^3})+\cdots\notag
\end{align}
where $(\mu_\alpha)_{\alpha=1}^\infty$ is a sequence of positive numbers determined below, and he shows the following lemma (see \cite[p. 426-427]{toeplitz1913}):
\begin{lema}\label{auxiltoe} If $\mu_\alpha>0$ are chosen so that the series $\sum\mu_\alpha$ is convergent, then the quadratic form \eqref{formatoe} is completely bounded on $\mathcal{S}$.
\end{lema}
Moreover, the sum of the moduli of all coefficients of the form $C(x)$ is
$\sum 2^\alpha\mu_\alpha$. It is easy to choose $\mu_\alpha$ so that $\sum\mu_\alpha<\infty$ and $\sum 2^\alpha\mu_\alpha=\infty$  (for example, $\mu_\alpha=\frac1{\alpha^2}$, $\mu_\alpha=2^{-\alpha}$, etc.). 
Thus, the function 
$$
F(x)=C(e^{2\pi i x})\quad (x\in\mathbb{T}^\omega)
$$
constructed with these $\mu_\alpha$ is our first announced counterexample.
 
\smallskip
\noindent\textbf{Littlewood's quadratic forms.}
From \cite[p. 171-173]{little1930} and \cite[p. 609-612]{bohnen1931} we can get a variety of counterexamples that generalize the preceding one, based on quadratic forms on
$\mathcal{S}$ for which not all the coefficients are real.

For example, let $N>2$ be a fixed integer, and consider the infinite collection of matrices
\begin{align*}
M_1&=\left(e^{2\pi i\frac{rs}N}\right)_{N\times N}, \quad r,s=1,\ldots,N,\\
M_\mu&=\left(e^{2\pi i\frac{rs}N}\cdot M_{\mu-1}\right)_{N^\mu\times N^\mu}, \quad r,s=1,\ldots,N, \quad \text{if $\mu> 1$}.
\end{align*}
All entries in $M_\mu$ are $N$th roots of unity, and $M_\mu$ is an unitary matrix for all $\mu\in\mathbb{N}$.
Let us denote by $M_\mu\bigl(x_1^{(\mu)},\ldots x_{N^\mu}^{(\mu)} \bigr)$ the quadratic form associated with the matrix $M_\mu$ and the variables of a generic point $x\in\mathcal{S}$ on which it acts, and then define the quadratic form in infinitely many variables
\begin{align*}
\begin{split}
M(x)&=N^{-3/2}M_1(x_1,\ldots,x_N)+\frac14 N^{-3}M_2(x_{N+1},\ldots,x_{N+N^2})
\\&\qquad +\frac19 N^{-9/2}M_3(x_{N+N^2+1},\ldots,x_{N+N^2+N^3})+\cdots
\end{split}\\
&=\sum_{\mu=1}^\infty \frac{N^{-3\mu/2}}{\mu^2}M_\mu\bigl(x_1^{(\mu)},\ldots x_{N^\mu}^{(\mu)} \bigr).
\end{align*}
According to Lemma \ref{formaunit} we have
$$
\Bigl|M_\mu\bigl(x_1^{(\mu)},\ldots x_{N^\mu}^{(\mu)} \bigr)\Bigr|\le N^{3\mu/2},
$$ 
so that
$$
|M(x)|\le\sum_{\mu=1}^\infty\frac1{\mu^2}<\infty.
$$
Thus, $M(x)$ is completely bounded and, by Corollary \ref{cuadratica} (b), it belongs to the class $C^{(\infty}$ on $\mathcal{S}$. But, if we denote 
$M(x)=\sum_{m,n=1}^\infty a_{mn}x_mx_n$, since all the moduli of the nonzero coefficients are equal to 1, we have
$$
\sum_{m,n=1}^\infty |a_{mn}|=N^{-3/2}\cdot N^2+\frac14 N^{-3}\cdot N^4+\frac19 N^{-9/2}\cdot N^6+\ldots=
\sum_{j=1}^\infty\frac{N^{j/2}}{j^2}=\infty
$$
and so the Fourier series of the function $G(x)=M(e^{2\pi i x})$, $x\in\mathbb{T}^\omega$, diverges absolutely. 

Bohnenblust and Hille \cite[p. 608-614]{bohnen1931} generalized the results of Littlewood to $m$-variate forms  ($m> 2$). This would provide new counterexamples, this time based on $m$-variate forms ($m> 2$) in infinitely many variables.

\section*{Acknowledgments}
The authors are grateful to Professor  A. D. Bendikov for the suggested ideas and the careful reading of the manuscript. 

Luz Roncal was supported by the Basque Government through the
BERC 2018--2021 program, by Spanish Ministry of Economy and Competitive\-ness MINECO through
BCAM Severo Ochoa excellence accreditation SEV-2013-0323, through project  MTM2015-65888-C04-4-P, and through project MTM2017-82160-C2-1-P funded by (AEI/FEDER, UE) and acronym \\``HAQMEC'', and by a 2017 Leonardo grant for Researchers and Cultural Creators, BBVA Foundation. The Foundation accepts no responsibility for the opinions, statements and contents included in the project and/or the results thereof, which are entirely the responsibility of the authors.


\begin{thebibliography}{999}



\bibitem{bendikov}
\textsc{Bendikov, A. D.},
\textit{Potential theory on infinite-dimensional abelian groups},  de Gruyter Stud. Math. 21, de Gruyter, Berlin, 1995.

\bibitem{bendikov86}
\textsc{Bendikov, A. D. and Pavlov, I. V.}, Boundedness of a class of vector-valued multiplier operators in $L_p(T^\infty)$. 
\textit{Siberian Math. J.} \textbf{27} (1986), 1--7. 

\bibitem{bendikov2005}
\textsc{Bendikov, A. D.  and  Saloff-Coste, L.}, Spaces of smooth functions and distributions on infinite-dimensional compact groups, \textit{Journal of Functional Analysis} \textbf{218} (2005), 168--218.

\bibitem{bohnen1931}
\textsc{Bohnenblust, H. F. and Hille, E.}, On the absolute convergence of Dirichlet series, \textit{Annals of Mathematics} (2)  32 (1931), 600--622.


\bibitem{bohr1925}
\textsc{Bohr, H.}, Zur Theorie der fast periodischen Funktionen, II Teil, \textit{Acta Mathematica}
\textbf{46} (1925), 101--214.

\bibitem{bromwich}
\textsc{Bromwich, T. J. i'A.}, \textit{An Introduction to the Theory of Infinite Series}, Macmillan, London, 1908.


\bibitem{edwards}
\textsc{Edwards, R. E.}, \textit{Fourier Series: A Modern Introduction I, II}, 2nd ed., Springer, New York, 1979.


\bibitem{hardy17}
\textsc{Hardy, G. H.}, On the convergence of certain multiple series, \textit{Proceedings of the Cambridge Philosophical Society} \textbf{19} (1917), 86--95.


\bibitem{hilbert1912}
\textsc{Hilbert, D.}, \textit{Grundz\"uge einer allgemeinen Theorie der linearen Integralgleichungen}, Fortschritte der Mathematischen Wissenschaften in Monographien, Heft 3, B. G. Teubner, Leipzig und Berlin, 1912.


\bibitem{little1930}
\textsc{Littlewood, J. E.}, On bounded bilinear forms in an infinite number of variables, \emph{Quarterly Journal of Mathematics (Oxford Series)} \textbf{1} (1930), 164--174.

\bibitem{moricz2011}
\textsc{M\'oricz, F.}, On the convergence of double integrals and
a generalized version of Fubini's theorem on successive integration, \textit{Acta Sci. Math. (Szeged)} 78 (2012), 469--487. 


\bibitem{rudinfa}
\textsc{Rudin, W.}, \textit{Functional analysis}, 2nd ed., McGraw-Hill, Singapore, 1991.

\bibitem{steinweiss}
\textsc{Stein, E. M. and Weiss, G.}, \textit{Introduction to Fourier analysis on euclidean spaces}, Princeton Univ. Press,  Princeton, NJ, 1971.

\bibitem{toeplitz1913}
\textsc{Toeplitz, O.}, \"Uber eine bei den Dirichletschen Reihen auftretende Aufgabe aus der Theorie der Potenzreihen von unendlichvielen Ver\"anderlichen, \emph{Nachr. Ges. Wiss. G\"ottingen Math.-Phys. Klasse} 1913, Heft 3,  417--432.

\end{thebibliography}
\end{document}